\documentclass[11pt,reqno]{article}
\usepackage{caption}
\usepackage{subcaption}
\usepackage{graphicx,amsmath,amssymb,setspace,color,xcolor}
\usepackage{bm,bbm}

\usepackage{srcltx, soul, cancel}
\usepackage{amsthm}
\usepackage[normalem]{ulem}

\newcommand{\be}{\begin{eqnarray}}
\newcommand{\ee}{\end{eqnarray}}
\newcommand{\by}{\begin{eqnarray*}}
\newcommand{\ey}{\end{eqnarray*}}
\newcommand{\en}{\end{enumerate}}
\newcommand{\bi}{\begin{itemize}}
\newcommand{\ei}{\end{itemize}}

\newtheorem{lemma}{Lemma}[section]

\newtheorem{remark}{Remark}[section]

\newtheorem{proposition}{Proposition}[section]
\newtheorem{example}{Example}[section]
\newtheorem{thm}{Theorem}[section]

\newcommand \la {\lambda}
\newcommand \tet {\theta}

\newcommand \gam {\gamma}
\newcommand \gamo {\gam_1}
\newcommand \gamt {\gam_2}
\newcommand \eps {\varepsilon}
\newcommand \del {\delta}

\newcommand \ome {\omega}

\newcommand \R {\mathbb{R}}
\newcommand \N {\mathbb{N}}
\newcommand \E {\mathbb{E}}

\newcommand \lan {\la_n}
\newcommand \tetn {\tet_n}
\newcommand \cn {c_n}
\newcommand \Yn {Y_n}
\newcommand \Fn {F_n}
\newcommand \Gn {G_n}

\newcommand \gn {g_n}

\newcommand \Vn {V_n}
\newcommand \Vdn {V_{D,n}}
\newcommand \VD {V_D}
\newcommand \bD {b_D}
\newcommand \bn {b_n}

\newcommand \sqrtn {\sqrt{n}}

\newcommand \mO {\mathcal{O}}
\newcommand \mC {\mathcal{C}}

\renewcommand{\theequation}{\arabic{section}.\arabic{equation}}

\numberwithin{equation}{section}

\newcommand{\sredm}[1]{\ifmmode\text{\xout{\ensuremath{\displaystyle \textcolor{red}{#1}}}}\else\sout{\textcolor{red}{#1}}\fi}

\begin{document}

\title{Optimal Dividend Problem: Asymptotic Analysis\thanks{This is the final version of the paper. To appear in {\it SIAM Journal on Financial Mathematics}}.}

\author{Asaf Cohen\footnote{Department of Mathematics, University of Michigan, Ann Arbor, Michigan 48109, USA, 
shloshim@gmail.com,
https://sites.google.com/site/asafcohentau/. The research of A. Cohen is supported by the National Science Foundation (DMS-2006305)
}
\and
Virginia R. Young\footnote{Department of Mathematics, University of Michigan, Ann Arbor, Michigan 48109, USA, vryoung@umich.edu.  V. R. Young thanks the Cecil J. and Ethel M. Nesbitt Chair of Actuarial Mathematics for partial financial support.}}

\maketitle

\begin{abstract}
We re-visit the classical problem of optimal payment of dividends and determine the degree to which the diffusion approximation serves as a valid approximation of the classical risk model for this problem.  Our results parallel some of those in B\"auerle \cite{B2004}, but we obtain sharper results because we use a different technique for obtaining them.  Specifically, B\"auerle \cite{B2004} uses probabilistic techniques and relies on convergence in distribution of the underlying processes.  By contrast, we use comparison results from the theory of differential equations, and these methods allow us to determine the rate of convergence of the value functions in question.

\medskip

{\bf  Keywords}:  Optimal dividend strategy, Cram\'er-Lundberg risk process, diffusion approximation, approximation error.

\medskip

{\bf AMS 2010 Subject Classification:} Primary: 91B30, 91G80, 93E20. Secondary: 45J05, 60G99, 90B20. 

\medskip

{\bf JEL Classification:} G22, C60.

\end{abstract}


\section{Introduction}

\setcounter{equation}{0}
\renewcommand{\theequation}
{1.\arabic{equation}}

A long-standing problem in insurance mathematics is optimal payment of dividends; see, for example, the survey by Avanzi \cite{A2009}.  In this paper, we are concerned about the degree to which the diffusion approximation serves as a valid approximation of the classical risk model when optimizing dividend payments.  Gerber, Shiu, and Smith \cite{GSS2008} address approximations to the dividend problem.  Also, B\"auerle \cite{B2004} considers the scaled dividend problem and proves that, as the scaling factor increases without bound, the value function converges to the one under the diffusion approximation.  

We, now, compare the model and techniques used and the results obtained by B\"auerle \cite{B2004} with the corresponding items in our work.  B\"auerle \cite{B2004} considers the case for which the rate of dividend payments is bounded, which corresponds to a classical continuous-time control problem.  While we use the same diffusion scaling, we, on the other hand, do not restrict the dividend rate to be bounded, which leads to a singular control problem.  Moreover, B\"auerle \cite{B2004} uses probabilistic techniques and relies on convergence in distribution of the underlying processes under Skorokhod's J1 topology. This procedure has two main components: (1) showing that any limit point on any arbitrary sequence of controls in the pre-limit problem does not perform better than the value function, and (2) showing that a candidate control for the pre-limit problem attains the value function.  For the first component, the proof uses tightness arguments, heavily utilizing the continuity of the underlying processes and the uniform boundedness of the control.  Due to the singular control considered here, the compactness and tightness arguments used by B\"auerle \cite{B2004} are not valid for our work.\footnote{To bypass this issue, one may use the {\it time-stretching method}, introduced  by Meyer and Zheng in \cite{MZ1984} and extensively used by Kurtz \cite{Kurtz1991, KuR1990Martingale, Kur1991Control}, Kushner and Martins \cite{Martins1990, Kushner1991}, Budhiraja \cite{BG2006, bud-ros2006, bud-ros2007}, Costantini and Kurtz \cite{Cos-Kur2018, Cos-Kur2019}, and Cohen \cite{coh2019, coh2019bro}.  The basic idea of this method is that time is stretched in accordance with the singular controls, so that the stretched underlying processes are uniformly Lipschitz continuous. Hence, tightness is attained under the J1 topology.  Finally, time is shrunk in accordance with the limiting stretched control. The payoff/cost of the shrunken processes are, then, compared with the value function.  Recently, Cohen \cite{COHEN_stretch} showed that by working with the weak-M1 topology instead of the J1topology, the time-transformations are unnecessary since they are embedded in the definition of the parametric representation, which define the weak-M1 topology. Hence, one may pursue a probabilistic proof under this topology to get convergence. However, this result does not provide a rate of convergence.}  By contrast, we rely on the closed-form expression of the limiting value function and use comparison results from the theory of differential equations; these methods allow us to determine the rate of convergence of the value functions in question; see Theorem \ref{thm:Vn_lim}.  Another issue we address is the rate of convergence of the difference between the value function in the pre-limit problem and the payoff function in the pre-limit problem using the optimal threshold from the limiting problem; see Theorem \ref{thm:Vdn_VD}. The importance of this result stems from the fact that the latter threshold admits a closed form, unlike for the pre-limit problem. Finally, B\"auerle \cite{B2004} also includes proportional reinsurance, but we omit reinsurance in the interest of simplicity.

The background for the comparison principle we are using is introduced in Cohen and Young \cite{CY2020}.  In that paper, the authors provide the rate of convergence of the probability of ruin in the Cram\'er--Lundberg model to its diffusion approximation. The present paper shows that this method can be elevated from the uncontrolled problem to an optimal control problem, which on top of this is a {\it singular} control problem.  Additionally, the comparison principles enables us to compare the value functions for different policies.

The remainder of the paper is organized as follows.  In Section \ref{sec:back}, we present the Cram\'er-Lundberg (CL) model and state results from Azcue and Muler \cite{AM2005} that we use to bound our value function.  Then, in Section \ref{sec:asymp}, we scale the CL model and show that, as the scaling factor increases without bound, the resulting value function converges to the one under the diffusion approximation, and we determine the rate of that convergence.  In that section, we also show that, if the insurer uses the optimal strategy under the diffusion approximation but for the scaled CL risk model, then doing so is $\eps$-optimal, and we specify the order of $\eps$ relative to the scaling factor.

\section{Classical risk model and Azcue and Muler \cite{AM2005}}\label{sec:back}

\setcounter{equation}{0}
\renewcommand{\theequation}
{2.\arabic{equation}}

\subsection{Cram\'er-Lundberg model with dividends}\label{sec:model}

Consider an insurer whose surplus process $U = \{ U_t \}_{t \ge 0}$ before paying dividends is described by a Cram\'er-Lundberg (CL) model, that is, the insurer receives premium income at a constant rate $c$ and pays claims according to a compound Poisson process.  Specifically,
\begin{equation}\label{eq:U}
U_t = x + ct - \sum_{i=1}^{N_t} Y_i,
\end{equation}
in which $U_0 = x \ge 0$ is the initial surplus, $N = \{N_t \}_{t \ge 0}$ is a homogeneous Poisson process with intensity $\la > 0$, and the claim sizes $Y_1, Y_2, \ldots$ are independent and identically distributed, positive random variables, independent of $N$.  All random variables are defined on a common probability space $\big( \Omega, \mathcal{F}, \mathbb{P} \big)$, with the natural filtration $\mathbb{F} = \{ \mathcal{F}_t \}_{t \ge 0}$ induced by the random variables.

Let $F_Y$ denote the common cumulative distribution function of $\{Y_i\}_{i \in \N}$, and assume that $Y$ has finite moment generating function $M_Y(s) = \E \big(e^{Ys} \big)$ for $s$ in a neighborhood of $0$; thus, $\E \big( Y^k \big) < \infty$ for $k = 1, 2, \dots$.  Finally, assume that the premium rate $c$ satisfies $c > \la \E Y$, and write $c = (1 + \tet) \la \E Y$, with positive relative risk loading $\tet > 0$.

The insurer pays dividends to its shareholders according to a process $D = \{ D_t \}_{t \ge 0}$, in which $D_t$ equals the cumulative dividends paid on or before time $t$, with $D_{0-} = 0$.  A dividend strategy $D$ is {\it admissible} if $D$ is non-decreasing and is predictable with respect to the filtration $\mathbb{F}$.

The surplus process $X = \{ X_t \}_{t \ge 0}$ after paying dividends is given by
\begin{equation}\label{eq:X}
X_t = x + ct - \sum_{i=1}^{N_t} Y_i - D_t,
\end{equation}
in which $X_0 = x \ge 0$ is the initial surplus.  Define the time of ruin $\tau$ by
\begin{equation}\label{eq:tau}
\tau = \inf\{ t\ge 0: X_t < 0\}.
\end{equation}
The insurer seeks to maximize the expected payoff of discounted dividends between now and the time of ruin, with corresponding value function $V$ defined by
\begin{equation}\label{eq:V}
V(x) = \sup_D \E \Bigg[ \int_0^\tau e^{- \del t} dD_t \, \Bigg| \, X_0 = x \Bigg],
\end{equation}
in which $\del > 0$ is the discount rate, and the supremum is taken over admissible dividend strategies.

Gerber \cite{G1969} shows that the optimal dividend strategy for the problem in \eqref{eq:V} is a {\it band} strategy.   A band strategy reduces to a {\it barrier} strategy if the initial surplus is less than the lowest band or if claim sizes are exponentially distributed.  From Theorem 2.45 of Schmidli \cite{S2008}, the value function $V$ is the minimal non-negative solution of the following integro-differential variational inequality on $\R^+$:
\begin{equation}\label{eq:IDeq}
\min \left[ (\la + \del) v(x) - c v_x(x) - \la \displaystyle \int_0^x v(x - y) dF_Y(y), \, v_x(x) - 1 \right] = 0.
\end{equation}
Furthermore, Theorem 2.39 of Schmidli \cite{S2008} states that $V$ is differentiable from the left and from the right on $(0, \infty)$, and \eqref{eq:IDeq} holds separately for both left- and right-derivatives.

\subsection{Results from Azcue and Muler \cite{AM2005}}\label{sec:AM2005}

We look for bounds for the value function $V$ as sub- and supersolutions of \eqref{eq:IDeq}, after we scale the CL model in Section \ref{sec:scale}.  To that end, define the operator $F$, acting on $u \in \mathcal{C}^1 \big(\R^+ \big)$, by the variational inequality in \eqref{eq:IDeq}, that is,
\begin{equation}\label{eq:F}
F \big(x, u(x), u_x(x), u(\cdot) \big) = \min \left[ (\la + \del) u(x) - c u_x(x) - \la \displaystyle \int_0^x u(x - y) dF_Y(y), \, u_x(x) - 1 \right].
\end{equation}
We say that a function $u \in \mathcal{C}^1 \big(\R^+ \big)$ is a {\it subsolution} of $F = 0$ if 
\[
F \big(x, u(x), u_x(x), u(\cdot) \big) \le 0,
\]
for all $x \ge 0$.  Similarly, we say that a function $v \in \mathcal{C}^1 \big(\R^+ \big)$ is a {\it supersolution} of $F = 0$ if
\[
F \big(x, v(x), v_x(x), v(\cdot) \big) \ge 0,
\]
for all $x \ge 0$.

We state results from Sections 4 and 5 of Azcue and Muler \cite{AM2005} as they apply to the model in this paper.  They state their results for viscosity sub- and supersolutions because they control surplus via reinsurance; however, their results also apply to our no-reinsurance model with classical sub- and supersolutions.  First, Azcue and Muler \cite{AM2005} prove a comparison result for functions that satisfy the following conditions:

\begin{enumerate}
\item{}  $u: \R^+ \to \R$ is locally Lipschitz.
\item{} If $0 \le x < y$, then $u(y) - u(x) \ge y - x$.
\item{} There exists a constant $k > 0$ such that $u(x) \le x + k$ for all $x \ge 0$.
\end{enumerate}

\noindent They note that the value function in \eqref{eq:V} satisfies these three conditions.

Proposition 4.2 of Azcue and Muler \cite{AM2005} shows that if $u$ is a subsolution and if $v$ is a supersolution of $F = 0$, both satisfying Conditions 1, 2, and 3, with $u(0) \le v(0)$, then $u \le v$ on $\R^+$.  Because the value function is a solution of $F = 0$ and, hence, a supersolution of $F = 0$, we will use this result in Section \ref{sec:error} to find a lower bound of the value function.

Proposition 5.1 of Azcue and Muler \cite{AM2005} shows that if $v$ is an absolutely continuous supersolution of $F = 0$ satisfying Condition 3, then $V \le v$ on $\R^+$.   We will use this result in Section \ref{sec:error} to find an upper bound of the value function.

\section{Asymptotic analysis}\label{sec:asymp}

\setcounter{equation}{0}
\renewcommand{\theequation}
{3.\arabic{equation}}

\subsection{Scaled model and diffusion approximation}\label{sec:scale}

Next, we scale the CL model by $n > 0$, as in Cohen and Young \cite{CY2020}.  In the scaled system, define $\lan = n \la$, so $n$ large is essentially equivalent to $\la$ large. Scale the claim severity by defining $\Yn = Y/\sqrtn$; thus, the variance of total claims during $[0, t]$ is invariant under the scaling, that is, $\lan \E \big(\Yn^2 \big) = \la \E \big(Y^2 \big)$ for all $n > 0$.  Finally, define the premium rate by $\cn = c + (\sqrtn - 1) \la \E Y$; thus, $\cn - \lan \E \Yn = c - \la \E Y$ is also invariant under the scaling.  We can also write $\cn = (\sqrtn + \tet) \la \E Y$, in which $c = (1 + \tet) \la \E Y$; moreover, we can write $\cn = (1 + \tetn) \lan \E \Yn$, in which $\tetn = \tet/\sqrtn$.  The diffusion approximation of the scaled surplus process before dividends is, therefore, 
\begin{equation}\label{eq:diff_approx}
\big( \cn - \lan \E \Yn \big)dt + \sqrt{\lan \E \big(\Yn^2 \big)} \, dB_t = \big( c - \la \E Y \big)dt + \sqrt{\la \E \big(Y^2 \big)} \, dB_t = \tet \la \E Y dt + \sqrt{\la \E \big(Y^2 \big)} \, dB_t,
\end{equation}
for some standard Brownian motion $B = \{ B_t \}_{t \ge 0}$.  Note that the diffusion approximation of the scaled CL model is {\it independent} of $n$.  See Iglehart \cite{I1969}, B\"auerle \cite{B2004}, Gerber, Shiu, and Smith \cite{GSS2008}, and Schmidli \cite{S2017} for more information about this scaling.

Let $\Vn$ denote the value function under the scaled CL model.   We wish to bound $\Vn$ by modifying $\VD$ via functions of order $\mO \big( n^{-1/2} \big)$, in which $\VD$ is the value function when uncontrolled surplus follows the diffusion approximation in \eqref{eq:diff_approx}.  Thus, first, we digress to compute $\VD$, which uniquely solves the following free-boundary problem:
\begin{equation}\label{eq:FBP_diff}
\begin{cases}
\del v(x) = \tet \la \E Y v_x(x) + \dfrac{1}{2} \, \la \E \big(Y^2 \big) v_{xx}(x), \qquad 0 \le x \le \bD, \\
v(0) = 0, \quad v_x(\bD) = 1, \quad v_{xx}(\bD) = 0,
\end{cases}
\end{equation}
with $\VD(x) = \VD(\bD) + (x - \bD)$ for $x > \bD$.  See Gerber and Shiu \cite{GS2004} for a thorough analysis of $\VD$'s problem.  Via a straightforward application of techniques from ordinary differential equations, we obtain
\begin{equation}\label{eq:VD}
\VD(x) = 
\begin{cases}
\dfrac{e^{\gamo x} - e^{-\gamt x}}{\gamo e^{\gamo \bD} + \gamt e^{-\gamt \bD}} \, , &\quad 0 \le x \le \bD, \vspace{1ex} \\
\VD(\bD) + (x - \bD), &\quad x > \bD,
\end{cases}
\end{equation}
in which $0 < \gamo < \gamt$ are given by
\begin{equation}\label{eq:gam1}
\gamo = \dfrac{1}{\la \E \big(Y^2 \big)} \left[ - \tet \la \E Y + \sqrt{\big( \tet \la \E Y \big)^2 + 2 \del \la \E \big( Y^2 \big)} \, \right],
\end{equation}
and
\begin{equation}\label{eq:gam2}
\gamt =  \dfrac{1}{\la \E \big(Y^2 \big)} \left[ \tet \la \E Y + \sqrt{\big( \tet \la \E Y \big)^2 + 2 \del \la \E \big( Y^2 \big)} \, \right] ,
\end{equation}
and the free boundary $\bD$ equals
\begin{equation}\label{eq:bD}
\bD = \dfrac{2}{\gamo + \gamt} \, \ln \left( \dfrac{\gamt}{\gamo} \right).
\end{equation}
By using the expression for $\bD$ in \eqref{eq:bD}, we rewrite $\VD$ as follows:
\begin{equation}\label{eq:VD_1}
\VD(x) =
\begin{cases}
\dfrac{1}{\gamo + \gamt} \left( \dfrac{\gamo}{\gamt} \right)^{\frac{\gamo - \gamt}{\gamo + \gamt}} \Big( e^{\gamo x} - e^{-\gamt x} \Big), &\quad 0 \le x \le \bD, \vspace{2ex} \\
\dfrac{\tet \la \E Y}{\del} + \big( x - \bD \big), &\quad x > \bD.
\end{cases}
\end{equation}
From the second line in \eqref{eq:VD_1}, we observe
\begin{equation}\label{eq:VD_bD}
\VD \big( \bD \big) = \dfrac{\tet \la \E Y}{\del} \, ,
\end{equation}
the present value of a continuous perpetuity, discounted at rate $\del$, paying at the rate $\tet \la \E Y$, the risk loading in the premium.  Also, $V_D(x) > x$ for all $x > 0$, from which it follows that
\begin{equation}\label{eq:bD_ineq}
\dfrac{\tet \la \E Y}{\del} > \bD,
\end{equation}
an inequality that will be useful later.

\begin{remark}
Because the diffusion in \eqref{eq:diff_approx} approximates the CL risk process in \eqref{eq:U} with $\la$, $Y$, and $c$ replaced by $\lan$, $\Yn$, and $\cn$, respectively, researchers often say that $\VD$ approximates $\Vn$. In Theorem {\rm \ref{thm:Vn_lim}} in the next section, we quantify the degree to which $\VD$ approximates $\Vn$.  \qed
\end{remark}




\subsection{Approximating $\Vn$ by $\VD$ to order $\mO \big( n^{-1/2} \big)$}\label{sec:error}

In this section, we bound $\Vn$ by modifying $\VD$ via functions of order $\mO \big( n^{-1/2} \big)$ and by using Propositions 4.2 and 5.1 of Azcue and Muler \cite{AM2005}, as they apply to the scaled problem.  Note that $\VD$ plus or minus a constant satisfies the three conditions of Azcue and Muler \cite{AM2005} that we list in Section \ref{sec:AM2005}.

Throughout this section, let $\Fn$ denote the operator in \eqref{eq:F}, with $c$, $\la$, and $Y$ replaced by $\cn$, $\lan$, and $\Yn$, respectively, and write $\Fn$ as follows:
\begin{align}\notag
\Fn \big(x, u(x), u_x(x), u(\cdot) \big) = \min \Big\{\Gn \big(x, u(x), u_x(x), u(\cdot) \big), \, u_x(x) - 1 \Big\},
\end{align}
in which the operator $\Gn$ is as defined by
\begin{align}
\Gn \big(x, u(x), u_x(x), u(\cdot) \big) =
(n \la + \del)u(x) - \la \big(\sqrtn + \tet \big) \E Y \, u_x(x) - n \la \int_0^{\sqrtn x} u \Big(x - \frac{t}{\sqrtn}\Big) dF_Y(t).
\label{eq:Gn_def}
\end{align}
Recall that $\Fn$ evaluated at the value function $\Vn$ is identically $0$.

In the next proposition, we modify $\VD$ by a constant of order $\mO \big( n^{-1/2} \big)$ to obtain a lower bound of $\Vn$.  In Appendix \ref{app:A}, we present the background calculation that inspired this bound.

\begin{proposition}\label{prop:lower_bnd}
Assume there exists $\eps > 0$ such that
\begin{equation}\label{eq:Y3_bnd}
\E\big(e^{\eps Y}\big) < \infty.
\end{equation}
Then, there exists $q = q(\eps) > 0$ and $N = N(\eps)>0$, such that, for all $n > N$ and all $x \ge 0$,
\begin{align}\label{eq:lower_bound_Vn}
\VD(x) - \dfrac{q}{\sqrtn} \le \Vn(x).
\end{align}
\end{proposition}

\begin{proof}
First, note that
\begin{equation}\label{eq:order0}
\VD(0) - \dfrac{q}{\sqrtn} = - \, \dfrac{q}{\sqrtn} < 0 \le \Vn(0).
\end{equation}
Thus, by Proposition 4.2 in Azcue and Muler \cite{AM2005}, to prove inequality \eqref{eq:lower_bound_Vn}, it suffices to show that there exists $q > 0$ such that $\VD - q/\sqrtn$ is a subsolution of $\Fn = 0$.

$\Gn$ is linear with respect to $u$, $u_x$, and $u(\cdot)$; thus,
\begin{align}
\Gn \big(x, \VD(x) - q/\sqrtn, \VD'(x), \VD(\cdot) - q/\sqrtn \, \big) &= \Gn \big(x, \VD(x), \VD'(x), \VD(\cdot) \big) - \dfrac{q}{\sqrtn} \, \Gn (x, 1, 0, 1) \notag \\
&= \Gn \big(x, \VD(x), \VD'(x), \VD(\cdot) \big) - \dfrac{q}{\sqrtn} \, \big( \del + n \la S_Y(d) \big),
\label{eq:Gn_Umq}
\end{align}
in which $d = \sqrtn x$ and $S_Y(d):=\mathbb{P}(Y > d)$.  Note that there is $N=N(\eps)>0$ such that,
\begin{equation}\label{eq:1}
\E\Big(
Y^3 \, e^{\frac{\gamt}{\sqrt{N}}Y} 
\Big) < \infty.
\end{equation}
Now, from \eqref{eq:Gn_bnd} in the appendix, for $n \ge N$, we can bound $\Gn$ evaluated at $\VD - q/\sqrtn$ on $[0, \bD]$ as follows:
\begin{align}
\Gn \big(x, \VD(x) - q/\sqrtn, \VD'(x), \VD(\cdot) - q/\sqrtn \, \big) &\le  \dfrac{\la A}{\sqrtn} - \dfrac{q}{\sqrtn} \, \big( \del + n \la S_Y(d) \big) < \dfrac{\la A -  q \del}{\sqrtn} \, ,
\label{eq:GnUmq_bnd}
\end{align}
in which the positive constant $A$ is defined in \eqref{eq:Gn1_bnd}.  Choose $q = q(N)$ so that
\begin{equation}\label{eq:q}
q \ge \dfrac{\la A}{\del}.
\end{equation}
(Because $A$ in \eqref{eq:Gn1_bnd} depends on $N$, $q$'s lower bound in \eqref{eq:q} depends on $N$.)  Then, $\la A  - q \del$ is non-positive, and inequality \eqref{eq:GnUmq_bnd} implies that $\Gn$ evaluated at $\VD - q/\sqrtn$ is negative on $[0, \bD]$.

Because $\VD'(x) = 1$ for all $x > \bD$, it follows that, for all $x \ge 0$ and for all $n > \max\big(N, q^2 \big)$,
\begin{align*}
\Fn \big(x, \VD(x) - q/\sqrtn, \VD'(x), \VD(\cdot) - q/\sqrtn \, \big) \le 0 = \Fn \big(x, \Vn(x), \Vn'(x), \Vn(\cdot) \big).
\end{align*}
We have shown that $\VD - q/\sqrtn$ is a subsolution of $\Fn = 0$, and Proposition 4.2 in Azcue and Muler \cite{AM2005} implies the bound of $\Vn$ in \eqref{eq:lower_bound_Vn}.
\end{proof}

In the next proposition, we provide an upper bound of $\Vn$, and we use Proposition 5.1 in Azcue and Muler \cite{AM2005} to prove the proposition.

\begin{proposition}\label{prop:upper_bnd}
Assume there exists $\eps > 0$ such that
\begin{equation}\label{eq:Zd_bnd}
\sup \limits_{d \ge 0} \E \big( e^{\eps(Y - d)} \, \big| \, Y > d \big) < \infty.
\end{equation}
Then, there exists $p =p(\eps)> 0$ and $N'=N'(\eps)>0$ such that, for all $n \ge N'$ and all $x \ge 0$,
\begin{align}\label{eq:upper_bound_Vn}
\Vn(x) \le \VD(x) + \dfrac{p}{\sqrtn}.
\end{align}
\end{proposition}

\begin{proof}
By Proposition 5.1 of Azcue and Muler \cite{AM2005}, because $\VD + p/\sqrtn$ is absolutely continuous and because $\VD(x) + p/\sqrtn \le x + k$ for some $k > 0$ and for all $x \ge 0$, to prove the bound in \eqref{eq:upper_bound_Vn}, it is enough to show that $\VD + p/\sqrtn$ is a supersolution of $\Fn = 0$.

First, notice that by the condition \eqref{eq:Zd_bnd}, there is $\tilde N > 0$, such that
\begin{equation}\label{eq:Zd_bnd2}
\sup_{d \ge 0} 
\E\Big(
(Y-d)^2 \, e^{\frac{\gamt}{\sqrt{\tilde N}}(Y-d)} \, \Big| \, Y > d
\Big) < \infty.
\end{equation}
and 
\begin{equation}\label{eq:2}
\E\Big(
Y^3 \, e^{\frac{\gamt}{\sqrt{\tilde N}}Y} 
\Big) < \infty.
\end{equation}
Second, evaluate $\Gn$ at $\VD + p/\sqrtn$ on $[0, \bD]$.  Let $d = \sqrtn x$ and $C = \big( \gamo + \gamt \big) \left( \frac{\gamt}{\gamo} \right)^{\frac{\gamo - \gamt}{\gamo + \gamt}}$; then, via a calculation similar to the one in Appendix \ref{app:A}, we have
\begin{align}\label{eq:Gn_upper}
&\Gn \Big( x, \VD(x) + \frac{p}{\sqrtn}, \VD'(x), \VD(\cdot) + \frac{p}{\sqrtn} \Big) = \Gn \big( x, \VD(x), \VD'(x), \VD(\cdot) \big) + \frac{p}{\sqrtn} \, \Gn(x, 1, 0, 1)  \notag \\
&= \frac{\la}{2C\sqrtn} \, \int_0^1 (1 - \ome)^2 \left\{ \gamo^3 e^{\gamo x} \E\Big(Y^3e^{\frac{- \gamo \ome}{\sqrtn}Y} \Big) + \gamt^3 e^{-\gamt x} \E\Big(Y^3e^{\frac{\gamt \ome}{\sqrtn}Y} \Big) \right\} d\ome \notag \\
&\quad + \frac{\la S_Y(d)}{C} \, \bigg\{ - \, \sqrtn \big(\gamo + \gamt \big) \, \E(Y - d|Y > d) \notag \\
&\qquad\qquad\qquad \left. + \int_0^1 (1 - \ome)  \, \E\Big[ \gamo^2 (Y-d)^2 \, e^{-\frac{\gamo \ome}{\sqrtn}(Y-d)}  - \gamt^2 (Y-d)^2 \, e^{\frac{\gamt \ome}{\sqrtn} (Y-d)}  \, \Big| \, Y > d \Big] d\ome \right\} \notag \\
&\quad + \dfrac{p}{\sqrtn} \, \big( \del + n \la S_Y(d) \big) \notag \\
&> \dfrac{\la S_Y(d)}{C} \bigg\{ \sqrtn \big( p C - (\gamo + \gamt)\E (Y - d|Y > d) \big) \notag \\
&\qquad\qquad\quad \left. - \int_0^1 (1 - \ome) \, \E \Big[ \gamt^2 (Y-d)^2 \, e^{\frac{\gamt \ome}{\sqrtn} (Y-d)} - \gamo^2 (Y-d)^2 \, e^{-\frac{\gamo \ome}{\sqrtn}(Y-d)} \, \Big| \, Y > d \Big] d\ome \right\}.
\end{align}
Choose $p$ so that $pC > (\gamo + \gamt) \sup_{d \ge 0} \E(Y - d|Y > d)$, this supremum is finite because of the bound in \eqref{eq:Zd_bnd2}.  Furthermore, the bound in \eqref{eq:2} implies that there exists $N' \ge \tilde N$ such that, if $n \ge N'$, then the expression in \eqref{eq:Gn_upper} is non-negative.  Also, $\VD' \ge 1$ on $[0, \bD]$, so $\Fn$ evaluated at $\VD + p/\sqrtn$ is non-negative on $[0, \bD]$.

Next, evaluate $\Gn$ at $\VD + p/\sqrtn$ on $(\bD, \infty)$.  Again, let $d = \sqrtn x$; then, after simplifying,
\begin{align*}
&\Gn \Big( x, \VD(x) + \frac{p}{\sqrtn}, \VD'(x), \VD(\cdot) + \frac{p}{\sqrtn} \Big) \\
&= \Gn \big( x, x - \bD, 1, \cdot - \bD \big) + \left( \dfrac{\tet \la \E Y}{\del} + \dfrac{p}{\sqrtn} \right) \Gn(x, 1, 0, 1) \\
&= \del \left\{ \big( x - \bD \big) + \dfrac{p}{\sqrtn}  \right\} + n \la S_Y(d) \left\{ \dfrac{\tet \la \E Y}{\del} - \bD + \dfrac{p - \E(Y - d|Y > d)}{\sqrtn} \right\} > 0,
\end{align*}
in which the inequality follows from \eqref{eq:bD_ineq} and from choosing $p > \sup_{d \ge 0} \E(Y - d|Y > d)$.  Also, $\VD' = 1$ on $(\bD, \infty)$; thus, $\Fn$ evaluated at $\VD + p/\sqrtn$ equals zero on $(\bD, \infty)$.

We have shown that $\VD + p/\sqrtn$ is a supersolution of $\Fn = 0$, and Proposition 5.1 in Azcue and Muler \cite{AM2005} implies the bound of $\Vn$ in \eqref{eq:upper_bound_Vn}.
\end{proof}

In the following theorem, we show that $\Vn$ converges to $\VD$ at a rate of order $\mO \big( n^{-1/2} \big)$.

\begin{thm}\label{thm:Vn_lim}
If \eqref{eq:Y3_bnd} and \eqref{eq:Zd_bnd} hold, then there exists $C' > 0$ such that, for all $n > \max(N, N')$ and $x \ge 0$, 
\begin{align}\label{ASAF1}
\big| \Vn(x) - \VD(x) \big| \le \dfrac{C'}{\sqrtn} \, .
\end{align}
\end{thm}

\begin{proof}
From Propositions \ref{prop:lower_bnd} and \ref{prop:upper_bnd}, it follows that 
$$
\VD(x) - \dfrac{q}{\sqrtn} \le \Vn(x) \le \VD(x) + \frac{p}{\sqrtn} \, ,
$$
with $q > 0$ and $p > 0$ given in the proofs of those propositions.  Subtracting $\VD(x)$ from each side yields
\begin{align}\notag
- \dfrac{q}{\sqrtn} \le \Vn(x) - \VD(x) \le \frac{p}{\sqrtn} \, .
\end{align}
Thus, if we set $C' = \max(q, p)$, inequality \eqref{ASAF1} follows.
\end{proof}

\begin{remark}
Our assumption that the moment generating function of $Y$ is finite in a neighborhood of $0$ stems from the fact that we estimate $G_n$ for the function $V_D$, which includes an exponential term.  Although our techniques cannot handle random variables with infinite moment generating functions $($for example, Pareto and lognormal$)$, observe that not only do we prove convergence, we also provide the {\rm rate} of convergence.   Specifically, Theorem {\rm \ref{thm:Vn_lim}} asserts that the rate of convergence of $\Vn$ to $\VD$ is of order $\mO\big(n^{-1/2} \big)$, and, moreover, that the convergences is {\rm uniform} over $x \in [0,\infty)$.  By using probabilistic techniques and relying on convergence in distribution of the underlying processes, Theorem $3.6$ in B\"auerle {\rm \cite{B2004}} proves the pointwise convergence $\lim\limits_{n \to \infty}\Vn(x) = \VD(x)$ without estimating the rate of convergence, but, as mentioned in the introduction, these techniques are only valid for the case of bounded rates of dividend payments.  We leave determining the rate of convergence in more general setups for future research.  \qed
\end{remark}

We end this section with an example in which we calculate $C'$ in Theorem \ref{thm:Vn_lim}.

\begin{example}\label{ex:gamma}
Let $Y \sim Gamma(2, 1)$ with probability density function $f_Y(y) = y e^{-y}$ for $y \ge 0$, $\la = 10$, $\tet = 0.07$, and $\del = 0.10$, which is the example Azcue and Muler {\rm \cite{AM2005}} consider in Section $10.1$ of their paper.  Azcue and Muler {\rm \cite{AM2005}} give the following value function for $n = 1;$ note that $V_1$ embodies a non-barrier band strategy:
\begin{align}\label{V1}
V_1(x) = 
\begin{cases}
x + 2.119, &\quad 0 \le x < 1.803, \\
0.0944 e^{-1.48825x} - 9.431 e^{-0.079355x} + 11.257 e^{0.039567x}, &\quad 1.803 \le x < 10.22, \\
x + 2.456, &\quad x \ge 10.22.
\end{cases}
\end{align}
For this example, $\gamo = 0.03894$, $\gamt = 0.08561$, and $\bD = 12.650$.  In \eqref{eq:1}, we may set $N = 1$, from which it follows that
\[
A = \frac{1}{6C} \left\{ \gamo^3 e^{\gamo \bD} \E\big(Y^3 \big) + \gamt^3 \, \E\Big(Y^3e^{\gamt Y} \Big) \right\} = 0.04651,
\]
and gives $q = \la A/\del = 100 A = 4.651$.  Also,
\[
\sup_{d \ge 0} \E(Y - d|Y > d) = \sup_{d \ge 0} \frac{2 + d}{1 + d} = 2,
\]
which implies that we can set
\[
p > 2 \left( \dfrac{\gamo}{\gamt} \right)^{\frac{\gamo - \gamt}{\gamo + \gamt}} = 2.687.
\]
It follows that $C' = \max(q, p) = 4.651$.  

The numerical scheme for computing $\Vn$ for $n \in \N$ is prescribed on pages $95$--$96$ in Schmidli {\rm \cite{S2008}}. For completeness, we describe it here.  From Gerber {\rm \cite{G1969}}, we know that the optimal policy is a band policy.  Now, for the initial capital $x = 0$, either $(a)$ dividends are paid, in which case $V_n(0) = \la(\sqrtn + \tet)/(n\la + \del);$ or, $(b)$ no dividends are paid immediately and there is a value $b_0 = \inf\{x>0: V_n'(x)=1\}>0$, and when surplus lies in the band $[0,b_0]$, no dividends are paid.  The value function in this case would be derived by taking a solution $u$ of $\Gn \big(x, u(x), u_x(x), u(\cdot) \big) = 0$ on $[0, b_0]$, with the initial condition $u(0) = 1$ and by setting $V_n(x) = u(x)/u_x(b_0)$ for $x\in[0,b_0]$. Then, one repeats this process on $[b_0,\infty)$.

In our case, assume that dividends are paid in the band that includes $x=0;$ then, by differentiating the integro-differential equation $\Gn = 0$ twice, one obtains the differential equation
\begin{align}
0=n\del  u(x)+2 \big(\sqrtn\del -\la\tet n \big)u_x(x) + \big(-3\la n-4\la\sqrtn+\del \big) u_{xx}(x)-2\la(\sqrtn+\tet)u_{xxx}(x),
\end{align}
which implies the general form of the solution $u(x)=\sum_{i=1}^3a_i e^{\alpha_i x}$.  By substituting this ansatz into $\Gn = 0$ and by using the initial condition $u(0) = 1$, one obtains, for $n = 1$,
\begin{align}\notag
u(x)=5.947983e^{0.039567x} - 5.058731e^{-0.079355x} + 0.110748e^{-1.48825x} .
\end{align}
However, the minimum of $u_x$ is attained at $x=0$. Hence, $u(x)/u(b_0)<1$ for any $b_0>0$. Therefore, we deduce that dividends are paid on the band that includes $x=0$. That is, $V_1(x) = x+ \la(\sqrtn+\tet)/(n\la + \del)$ in a neighborhood of $x = 0$. To find the first band's upper threshold $b_0$, we define for any $b > 0$ the function $V_1^b(x):[b,\infty) \to \R$ such that $V_1^b$ solves the integro-differential equation with the initial condition $V_1^b(b) = V_1(b) = b + \la(\sqrtn+\tet)/(n\la +\del)$. Then, we set 
$$
b_0=\underset{x\ge b}{\text{argmin}}\;(V_1^b)_x(x)=1,
$$
and  $b_1$ equals the value of $x$ for which  $(V_1^{b_0})_x(x)=1$. The latter is the upper bound of the second band from the bottom.  In our case, $b_0=1.80303$ and $b_1=10.2162$. Above $x=b_1$, it is always optimal to pay dividends. In conclusion, $V_1$ is given by \eqref{V1}. 

We repeated this procedure for $n = 4$, $9$, and $25$ and obtained the following value functions: 
\begin{align}\notag
V_{4}(x) = 
\begin{cases}
x + 1.66, &\quad 0 \le x < 0.63, \\
0.0441289 e^{-2.98829x} -10.07189 e^{-0.00823677x} + 10.86314 e^{0.0392537x}, &\quad 0.63 \le x < 10.8, \\
x + 1.799
, &\quad x \ge 10.8.
\end{cases}
\end{align}
\begin{align}\notag
V_{9}(x) = 
\begin{cases}
x + 0.94746, &\quad 0 \le x < 0.266, \\
0.072965 e^{-4.4883x} -10.2573 e^{-0.0834207x} + 10.84119 e^{0.0394149x}, &\quad 0.266 \le x < 13.24343, \\
x + 1.66343
, &\quad x \ge 13.24343.
\end{cases}
\end{align}
\begin{align}\notag
V_{25}(x) = 
\begin{cases}
x + 0.51043, &\quad 0 \le x < 0.105, \\
0.0441289 e^{-7.48831x} -10.51006 e^{-0.00842816x} + 10.86314 e^{0.0390648x}, &\quad 0.105 \le x < 12.11, \\
x + 1.537
, &\quad x \ge 12.11.
\end{cases}
\end{align}
The comparison between these three $V_n$'s and the respective bounds $V_D+p/\sqrtn$ and $V_D-q/\sqrtn$ is illustrated in Figure \ref{fig1}.
\end{example}

\subsection{$\mO \big( n^{-1/2} \big)$-optimality of using the barrier $\bD$ for the scaled CL model}\label{sec:bsn}

In this section, we show that using $\bD$ as a barrier strategy for the scaled CL model, in place of the optimal band strategy, is $\mO\big(n^{-1/2} \big)$-optimal.  Specifically, we show that there exists $C'' > 0$ and $N > 0$ such that $n > N$ implies
\[
\big| \Vn(x) - \Vdn(x) \big| < \dfrac{C''}{\sqrtn} \, ,
\]
for all $x > 0$, in which $\Vdn$ denotes the (expected) payoff function for the scaled problem when we use the barrier $\bD$.  From Lemma 2.48 of Schmidli \cite{S2008}, we know that there is a unique solution $\gn \in \mC^1(\R+)$ of the integro-differential equation $\Gn = 0$ with $\gn(0) = 1$.  Moreover, the proof of this lemma shows that $\gn$ is strictly increasing.  We use $\gn$ to construct an expression for $\Vdn$ as follows:
\begin{equation}\label{eq:Vdn}
\Vdn(x) = 
\begin{cases}
\dfrac{\gn(x)}{\gn'(\bD)}, &\quad 0 \le x \le \bD, \vspace{1ex} \\
\Vdn(\bD) + (x - \bD), &\quad x > \bD.
\end{cases}
\end{equation}
Note that $\Vdn \in \mC^1(\R^+)$ with $\Vdn'(\bD) = 1$.

We connect $\Vdn$ and $\Vn$ via $\VD$ because (1) $\Vdn$ and $\VD$ are (expected) payoff functions for two different problems (scaled CL model versus its diffusion approximation) but with the {\it same} barrier $\bD$, and (2) we have a demonstrated relationship between $\VD$ and $\Vn$ in \eqref{ASAF1} in Theorem \ref{thm:Vn_lim}.  We begin by proving a theorem that is parallel to Theorem \ref{thm:Vn_lim}, but, first, we prove a comparison lemma for $\Gn$ on $[0, \bD]$, which we use to prove the parallel theorem.

\begin{lemma}\label{lem:comp}
Suppose $u, v \in \mC^1([0, \bD])$ satisfying the following conditions:
\begin{enumerate}
\item[$(i)$] $u(0) \le v(0)$.
\item[$(ii)$] $\Gn \big(x, u(x), u_x(x), u(\cdot) \big) \le \Gn \big(x, v(x), v_x(x), v(\cdot) \big)$ for all $x \in (0, \bD]$.
\item[$(iii)$] $u_x(\bD) = v_x(\bD)$.
\end{enumerate}
Then, $u \le v$ on $[0, \bD]$.
\end{lemma}

\begin{proof}
Suppose, on the contrary, that $u(x) > v(x)$ for some value of $x \in (0, \bD]$.  Then, there exists $x_0 \in (0, \bD]$ at which $u - v$ achieves a positive maximum, with $u_x(x_0) = v_x(x_0)$.  Note that condition (iii) of the lemma ensures $u_x(x_0) = v_x(x_0)$ if $x_0$ equals the endpoint $\bD$.  Then,
\begin{align*}
0 &\le \Gn\big(x_0, v(x_0), v_x(x_0), v(\cdot) \big) - \Gn \big(x_0, u(x_0), u_x(x_0), u(\cdot) \big) \\
&= (n \la + \del)v(x_0) - \la \big(\sqrtn + \tet \big) \E Y \, v_x(x_0) - n \la \int_0^{\sqrtn x_0} v \Big(x_0 - \frac{t}{\sqrtn}\Big) dF_Y(t) \\
&\quad - (n \la + \del)u(x_0) + \la \big(\sqrtn + \tet \big) \E Y \, u_x(x_0) + n \la \int_0^{\sqrtn x_0} u \Big(x_0 - \frac{t}{\sqrtn}\Big) dF_Y(t) \\
&= (n\la + \del) \big( v(x_0) - u(x_0) \big) - n \la \int_0^{\sqrtn x_0} \left(v \Big(x_0 - \frac{t}{\sqrtn}\Big) - u \Big(x_0 - \frac{t}{\sqrtn}\Big)\right)dF_Y(t) \\
&= - \big(n\la S_Y(\sqrtn x_0) + \del \big) \big( u(x_0) - v(x_0) \big) \\
&\quad - n \la \int_0^{\sqrtn x_0} \left(\big(u(x_0) - v(x_0) \big) - \left( u \Big(x_0 - \frac{t}{\sqrtn}\Big) - v \Big(x_0 - \frac{t}{\sqrtn}\Big)\right) \right)dF_Y(t) \\
&< 0,
\end{align*}
in which the inequality follows because $u - v$ achieves a positive maximum at $x_0$. Thus, we have a contradiction, so $u \le v$ on $[0, \bD]$.
\end{proof}

In the following theorem, we use Lemma \ref{lem:comp} to show that $\Vdn$ converges to $\VD$ at a rate of order $\mO \big( n^{-1/2} \big)$.  Theorem 3.7 in B\"auerle \cite{B2004} proves the pointwise convergence $\lim\limits_{n \to \infty} \Vdn(x) = \VD(x)$ without estimating the rate of convergence.

\begin{thm}\label{thm:Vdn_VD}
If \eqref{eq:Y3_bnd} and \eqref{eq:Zd_bnd} hold, then there exists $C' > 0$ such that, for all $n > \max(N, N')$ and $x \ge 0$, 
\begin{align}\label{ASAF2}
\big| \Vdn(x) - \VD(x) \big| \le \dfrac{C'}{\sqrtn} \, .
\end{align}
\end{thm}

\begin{proof}
Note that $\Vdn \le \Vn$ on $\R^+$ because they are payoff functions for the same problem, and $\Vn$ is the maximum as the value function.   Also, from Proposition \ref{prop:upper_bnd}, there exists $p > 0$ such that $\Vn \le \VD + p/\sqrtn$ on $\R^+$; thus, $\Vdn \le \VD + p/\sqrtn$ on $\R^+$.

Next, compare $\VD - q/\sqrtn$ and $\Vdn$ on $[0, \bD]$ for $q > 0$ given in Proposition \ref{prop:lower_bnd}.  In the proof of that proposition, we show that $\Gn$ evaluated at $\VD - q/\sqrtn$ is negative on $[0, \bD]$, and the construction in \eqref{eq:Vdn} shows that $\Gn$ evaluated at $\Vdn$ is zero on $[0, \bD]$.  Because $u = \VD - q/\sqrtn$ and $v = \Vdn$ satisfy the conditions of Lemma \ref{lem:comp}, we deduce that $\VD - q/\sqrtn \le \Vdn$ on $[0, \bD]$.  Furthermore, because $\Vdn$ and $\VD - q/\sqrtn$ have slope identically equal to $1$ for $x \ge \bD$, we deduce $\VD - q/\sqrtn \le \Vdn$ on $\R^+$.  

Thus, if we set $C' = \max(q, p)$, as in the proof of Theorem \ref{thm:Vn_lim}, inequality \eqref{ASAF2} follows.
\end{proof}

The following theorem is the main result of this section.

\begin{thm}\label{thm:bD_eps_opt}
If \eqref{eq:Y3_bnd} and \eqref{eq:Zd_bnd} hold, then the barrier strategy with barrier $\bD$ is $\mO\big(n^{-1/2} \big)$-optimal for the scaled problem.  Specifically, then there exists $C'' > 0$ such that, for all $n > \max(N, N')$ and $x \ge 0$, 
\begin{align}\label{ASAF3}
\big| \Vn(x) - \Vdn(x) \big| \le \dfrac{C''}{\sqrtn} \, .
\end{align}
\end{thm}

\begin{proof}
If we set $C'' = 2C'$, then \eqref{ASAF3} follows from \eqref{ASAF1}, \eqref{ASAF2}, and the triangle inequality.  
\end{proof}

\begin{remark}
Theorem {\rm \ref{thm:bD_eps_opt}} proves the $\mO\big(n^{-1/2} \big)$-optimality of using the optimal barrier for the diffusion approximation $($namely, $\bD)$ in place of the optimal band strategy for the $n$-scaled problem.  This result supports the common practice in the mathematical finance and insurance literature of using the diffusion approximation in place of the classical risk model.   \qed
\end{remark}

Although Theorem \ref{thm:bD_eps_opt} proves that using the barrier strategy with barrier $\bD$ is $\mO\big(n^{-1/2} \big)$-optimal for $\Vn$'s problem, we do not know that the barriers of $\Vn$'s band strategy converge to $\bD$; in fact, it does not appear to be true generally.  Indeed, consider the Gamma example on pages 95 and 96 of Schmidli \cite{S2008}:\ $V'(x) = 1$ in a neighborhood of $x = 0$, so we hypothesize that, as $n \to \infty$, the smallest barrier goes to $0$.

We end this paper with an example:\ let $Y \sim Exp(1)$, which implies $\Yn$ is exponentially distributed with mean $1/\sqrtn$.  As is well known, the optimal dividend strategy for $\Vn$'s problem is a barrier strategy; see, for example, Chapter 10 of Gerber \cite{G1979} or Avanzi \cite{A2009}.  Moreover, we have explicit expressions for the value function $\Vn$ and for its corresponding barrier $\bn$, and it is the latter in which we are interested.  From equation (2.28) in Avanzi \cite{A2009}, the optimal barrier equals
\begin{equation}\label{eq:bsn_exp}
\bn = \dfrac{1}{r_1 + r_2} \ln \left( \dfrac{r_2^2\big(\sqrtn - r_2 \big)}{r_1^2 \big(\sqrtn + r_1 \big)} \right),
\end{equation}
in which
\begin{equation}\label{eq:r1}
r_1 = \dfrac{1}{2 \la \big(\sqrtn + \tet \big)} \left[ \sqrt{\big(\sqrtn \la \tet + \del \big)^2 + 4n\del \la} - \big(\sqrtn \la \tet - \del \big) \right],
\end{equation}
and
\begin{equation}\label{eq:r2}
r_2 = \dfrac{1}{2 \la \big(\sqrtn + \tet \big)} \left[ \sqrt{\big(\sqrtn \la \tet + \del \big)^2 + 4n\del \la} + \big(\sqrtn \la \tet - \del \big) \right].
\end{equation}
It is easy to see that
\[
\lim_{n \to \infty} r_1 = \dfrac{1}{2} \left[ \sqrt{\tet^2 + 4 \del /\la} - \tet \right] = \gamo,
\]
and
\[
\lim_{n \to \infty} r_2 = \dfrac{1}{2} \left[ \sqrt{\tet^2 + 4 \del /\la} + \tet \right] = \gamt,
\]
in which $\gamo$ and $\gamt$ are given in \eqref{eq:gam1} and \eqref{eq:gam2}, respectively.  Thus,
\[
\lim_{n \to \infty} \bn = \dfrac{1}{\gamo + \gamt} \ln \left( \dfrac{\gamt^2}{\gamo^2} \right) = \dfrac{2}{\gamo + \gamt} \ln \left( \dfrac{\gamt}{\gamo} \right) = \bD,
\]
as expected, and the rate of convergence is of order $\mO \big( n^{-1/2} \big)$. Indeed,
\begin{align}\notag
|b_n - b_D|
&\le \left|\dfrac{1}{r_1 + r_2} - \dfrac{1}{\gamo + \gamt} \right| \times
\left| \, \ln \left( \dfrac{r_2^2}{r_1^2} \right) + \ln \left( \dfrac{\sqrtn - r_2}{\sqrtn + r_1} \right) \right| \\ \notag
&\quad + \dfrac{1}{\gamo + \gamt} \times
\left| \, \ln \left( \dfrac{r_2^2 \big(\sqrtn - r_2 \big)}{r_1^2 \big(\sqrtn + r_1 \big)} \right) - \ln\left(\frac{\gamt^2}{\gamo^2}\right)\right|.
\end{align}
Because $\ln(r^2_2/r^2_1)$ converges to $\ln(\gamt^2/\gamo^2)$ as $n\to\infty$, it follows that $\big| \ln(r^2_2/r^2_1) \big|$ is uniformly bounded over $n$.  Also, $\ln \Big( \frac{\sqrtn - r_2}{\sqrtn + r_1} \Big)$ is of order $\mO \big( n^{-1/2} \big)$. Hence, it suffices to estimate the following terms:
\begin{align}\notag
\left|\dfrac{1}{r_1 + r_2}-\dfrac{1}{\gamo+\gamt}\right|
\qquad \text{and} \qquad 
\left| \, \ln \left( \dfrac{r_2^2 \big(\sqrtn - r_2 \big)}{r_1^2 \big(\sqrtn + r_1 \big)} \right)-\ln\left(\frac{\gamt^2}{\gamo^2}\right) \right|.
\end{align}
Starting with the first term, there exists a constant $C>0$ such that for any $n>0$, 
\begin{align}\notag
\left|\dfrac{1}{r_1 + r_2} - \dfrac{1}{\gamo+\gamt} \right|
&=
\la \left| \frac{\sqrtn + \tet}{\sqrtn \sqrt{(\la\tet + \del/\sqrtn \,)^2 + 4 \la \del}}-\frac{1}{\sqrt{(\la\tet)^2+4\la\del}} \right| \\ \notag
&\le
\la \left| \frac{1}{\sqrt{(\la\tet+\del/\sqrtn \,)^2+4\la\del}}-\frac{1}{\sqrt{(\la\tet)^2+4\la\del}}\right| \\ \notag
&\qquad+\frac{1}{\sqrtn} \cdot \frac{\la\tet}{\sqrt{(\la\tet+\del/\sqrtn \,)^2 + 4\la\del}} \\ \notag
&\le \frac{C}{\sqrtn} \, .
\end{align}
Next, 
\begin{align}\notag
&\left| \, \ln \left( \dfrac{r_2^2\big(\sqrtn - r_2 \big)}{r_1^2 \big(\sqrtn + r_1 \big)} \right) - \ln\left(\frac{\gamt^2}{\gamo^2}\right) \right|
\le 
\left| \, \ln\left(\frac{\sqrtn-r_2}{\sqrtn + r_1}\right) \right| + 2\left| \, \ln\left(\frac{r_1}{\gamo}\right) \right| + 2 \left| \, \ln\left(\frac{r_2}{\gamt}\right) \right|.
\end{align}
The first term on the right side is of order $\mO \big( n^{-1/2} \big)$. The estimations for the last two terms are similar; hence, we provide details only for the last one.
\begin{align}\notag
\left| \, \ln\left(\frac{r_2}{\gamt}\right) \right|
&= \left| \,  \ln\left(\frac{\sqrtn}{\sqrtn+\tet} \cdot \frac{\sqrt{(\la\tet+\del/\sqrtn \,)^2+4\del\la}+\la\tet-\del/\sqrtn}{\sqrt{(\la\tet)^2+4\del\la}+\la\tet}\right) \right| \\ \notag
&\le   \ln\left(1+\frac{\tet}{\sqrtn}\right)
+
\left| \, \ln\left(\frac{\sqrt{(\la\tet+\del/\sqrtn \,)^2+4\del\la}+\la\tet-\del/\sqrtn}{\sqrt{(\la\tet)^2+4\del\la}+\la\tet}\right) \right|.
\end{align}
The right side is bounded from above by $C\ln(1 + C/\sqrtn \,)$, for some positive constant $C$, independent of $n$, and this bound is of order $\mO \big( n^{-1/2} \big)$.

\appendix

\section{$\Gn$ evaluated at $\VD$ on $[0, \bD]$}\label{app:A}

\setcounter{equation}{0}
\renewcommand{\theequation}
{A.\arabic{equation}}

In this appendix, we present the calculations that inspired Proposition \ref{prop:lower_bnd}.  Recall that
\begin{align}\notag
&\Fn \big(x, u(x), u_x(x), u(\cdot) \big) = \min \Big\{\Gn \big(x, u(x), u_x(x), u(\cdot)\big), \, u_x(x) - 1 \Big\},
\end{align}
in which the operator $\Gn$ is defined in \eqref{eq:Gn_def}.  We now evaluate $\Gn$ at $\VD$ for $0 \le x \le \bD$.
\begin{align}\label{eq:Gn}
& \Gn \big( x, \VD(x), \VD'(x), \VD(\cdot) \big) \\ \notag 
& =  \la \left\{\int_0^\infty \left[ n \left(\VD(x) - \VD\Big(x - \frac{t}{\sqrtn}\Big) \right) - \sqrtn \, \E Y \, \VD'(x) + \left(\dfrac{\del}{\la} \, \VD(x) - \tet \E Y \VD'(x) \right) \right]dF_Y(t)\right.\\\notag
&\qquad \; \left. + \int_{\sqrtn x}^\infty  n \VD \Big(x - \frac{t}{\sqrtn}\Big)dF_Y(t)\right\},
\end{align}
in which we extend $\VD(x)$ to $x < 0$ via the first expression in \eqref{eq:VD}.  Note that $\VD(x) < 0$ for $x < 0$, which implies that the second integral above is non-positive.  In the first integral, we write $\VD(x) = \big(e^{\gamo x} - e^{-\gamt x} \big)/C$, in which $C$ equals
\[
C = \big( \gamo + \gamt \big) \left( \dfrac{\gamt}{\gamo} \right)^{\frac{\gamo - \gamt}{\gamo + \gamt}}.
\]
The first integral is linear in $\VD$ and, hence, equals the difference of two integrals:\ one with $e^{\gamo x}/C$ replacing $\VD(x)$, and the second with $e^{-\gamt x}/C$ replacing $\VD(x)$.  We obtain the second of these integrals from the first one by substituting $- \gamt$ for $\gamo$; thus, we show the details only for $\gamo$.
\begin{align}
\notag
&\frac{n}{C} \int_0^\infty \left[e^{\gamo x} - e^{\gamo \big(x - \frac{t}{\sqrtn} \big)} - \frac{1}{\sqrtn} \, \E Y\gamo e^{\gamo x} + \dfrac{1}{n}\left(\dfrac{\del}{\la} \, e^{\gamo x} - \tet \E Y \gamo e^{\gamo x} \right)\right]dF_Y(t) \\
\notag
&\quad = - \, \frac{n}{C} \, e^{\gamo x} \int_0^\infty \left[e^{-\gamo t/\sqrtn} - \left(1 - \frac{\gamo t}{\sqrtn} + \frac{\gamo^2t^2}{2n} \right)\right]dF_Y(t)\\
\notag
&\quad = \frac{\gamo^3}{2C\sqrtn} \, e^{\gamo x}\int_0^1(1-\ome)^2\, \E\Big[Y^3e^{\frac{-\gamo \ome}{\sqrtn}Y} \Big]d\ome,
\end{align}
in which the first and second equalities, respectively, follow from the identities
\begin{align}\notag
\frac{\del}{\la} - \tet \E Y \gamo = \frac{\gamo^2 \, \E \big( Y^2 \big)}{2},
\end{align}
and 
\begin{align}\label{eq:ex}
e^x = 1 + x + \dfrac{x^2}{2!} + \cdots + \dfrac{x^n}{n!} + \dfrac{x^{n+1}}{n!} \int_0^1 (1 - \ome)^n e^{\ome x} d\ome.
\end{align}
Similar analysis for $e^{-\gamt x}/C$ yields that the first integral on the right side of \eqref{eq:Gn} equals
\begin{align}\label{eq:Gn1}
\frac{\gamo^3}{2C\sqrtn} \, e^{\gamo x}\int_0^1 (1-\ome)^2 \, \E\Big[Y^3e^{\frac{-\gamo \ome}{\sqrtn}Y} \Big] d\ome + \frac{\gamt^3}{2C\sqrtn} \, e^{-\gamt x} \int_0^1 (1-\ome)^2 \, \E\Big[Y^3e^{\frac{\gamt \ome}{\sqrtn}Y} \Big] d\ome.
\end{align}
Because of the bound in \eqref{eq:Y3_bnd}, we can rewrite and bound the first integral in \eqref{eq:Gn} on $[0, \bD]$ as follows:\ for $n \ge N$, in which $N$ is such that inequality \eqref{eq:1} holds,
\begin{align}
&\frac{1}{2C\sqrtn} \, \int_0^1 (1 - \ome)^2 \left\{ \gamo^3 e^{\gamo x} \E\Big(Y^3e^{\frac{- \gamo \ome}{\sqrtn}Y} \Big) + \gamt^3 e^{-\gamt x} \E\Big(Y^3e^{\frac{\gamt \ome}{\sqrtn}Y} \Big) \right\} d\ome \notag \\
&\le \frac{1}{2C\sqrtn} \, \int_0^1 (1 - \ome)^2 \left\{ \gamo^3 e^{\gamo \bD} \E\big(Y^3 \big) + \gamt^3 \, \E\Big(Y^3e^{\frac{\gamt}{\sqrtn}Y} \Big) \right\} d\ome \notag \\
&= \frac{1}{6C\sqrtn} \left\{ \gamo^3 e^{\gamo \bD} \E\big(Y^3 \big) + \gamt^3 \, \E\Big(Y^3e^{\frac{\gamt}{\sqrt{N}}Y} \Big) \right\} =: \dfrac{A}{\sqrtn} \, .  \label{eq:Gn1_bnd}
\end{align}
Thus, for $n \ge N$ and $x \in [0, \bD]$, we have
\begin{align}
&\Gn \big( x, \VD(x), \VD'(x), \VD(\cdot) \big) \le \dfrac{\la A}{\sqrtn} \, .
\label{eq:Gn_bnd}
\end{align}

\vspace{5pt}

\noindent{\bf Acknowledgement.} We thank an anonymous AE and two referees for their suggestions, which helped us to improve our paper.

\footnotesize
\bibliographystyle{abbrv} 
\bibliography{refs} 

\begin{figure}[p]
  \centering
  \includegraphics[width=.45\textwidth]{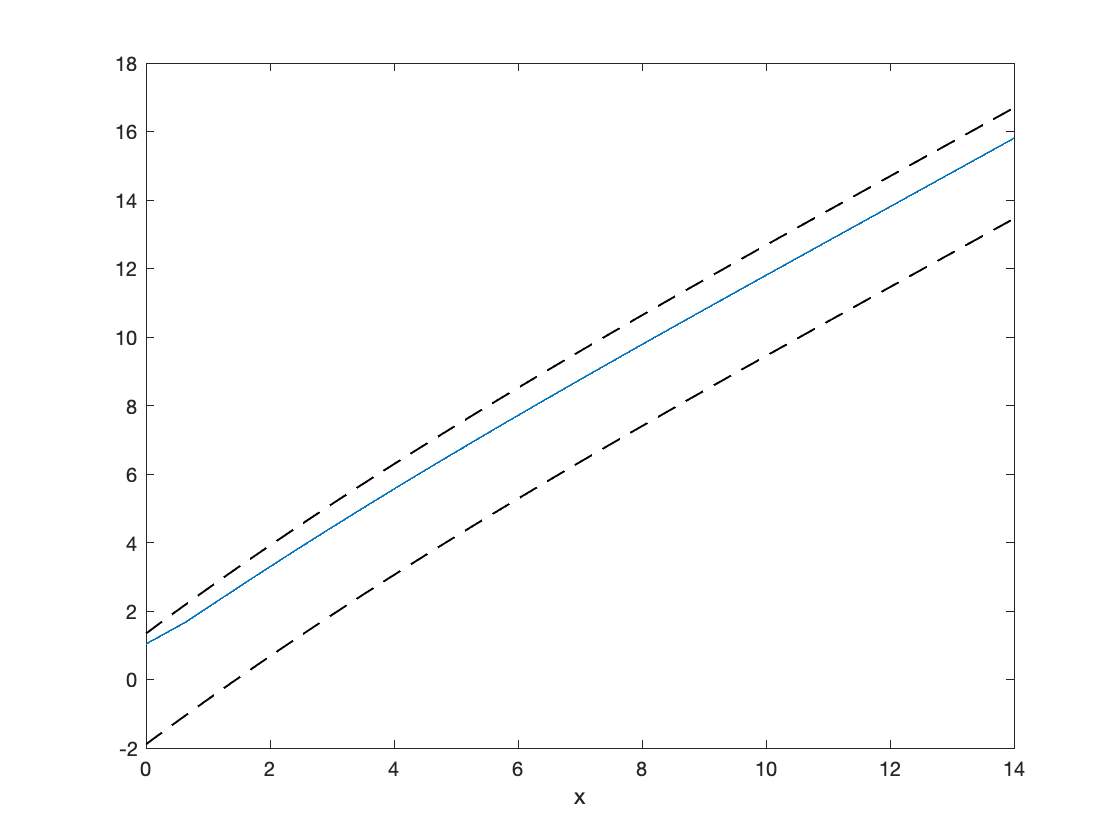}
  \hspace{1cm}
  \includegraphics[width=.45\textwidth]{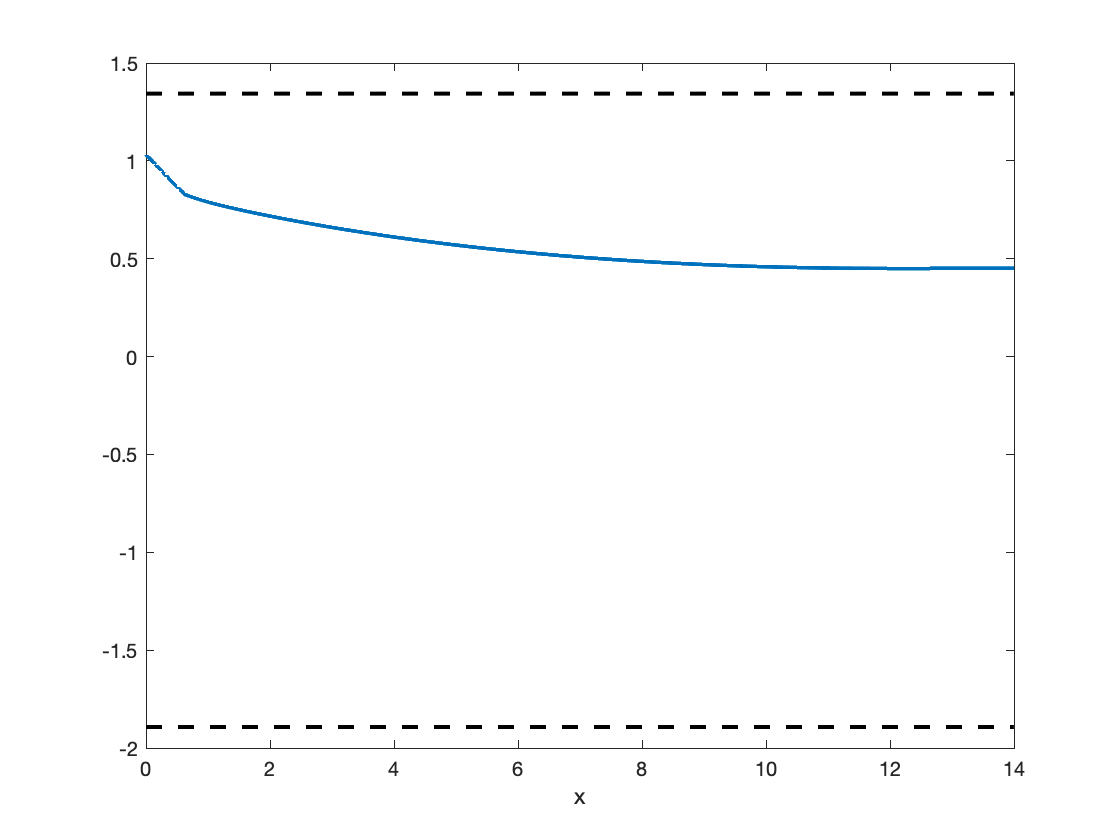}

\caption*{$n=4$}

  \includegraphics[width=.45\textwidth]{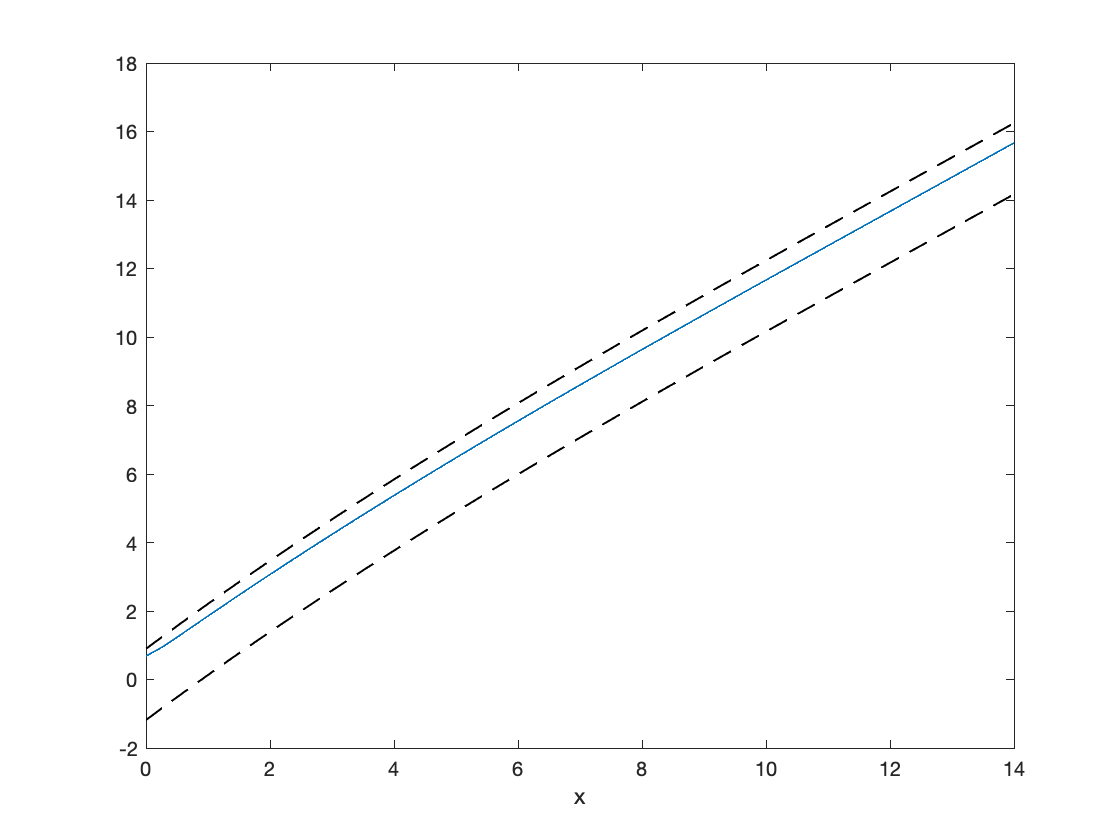}
    \hspace{1cm}
  \includegraphics[width=.45\textwidth]{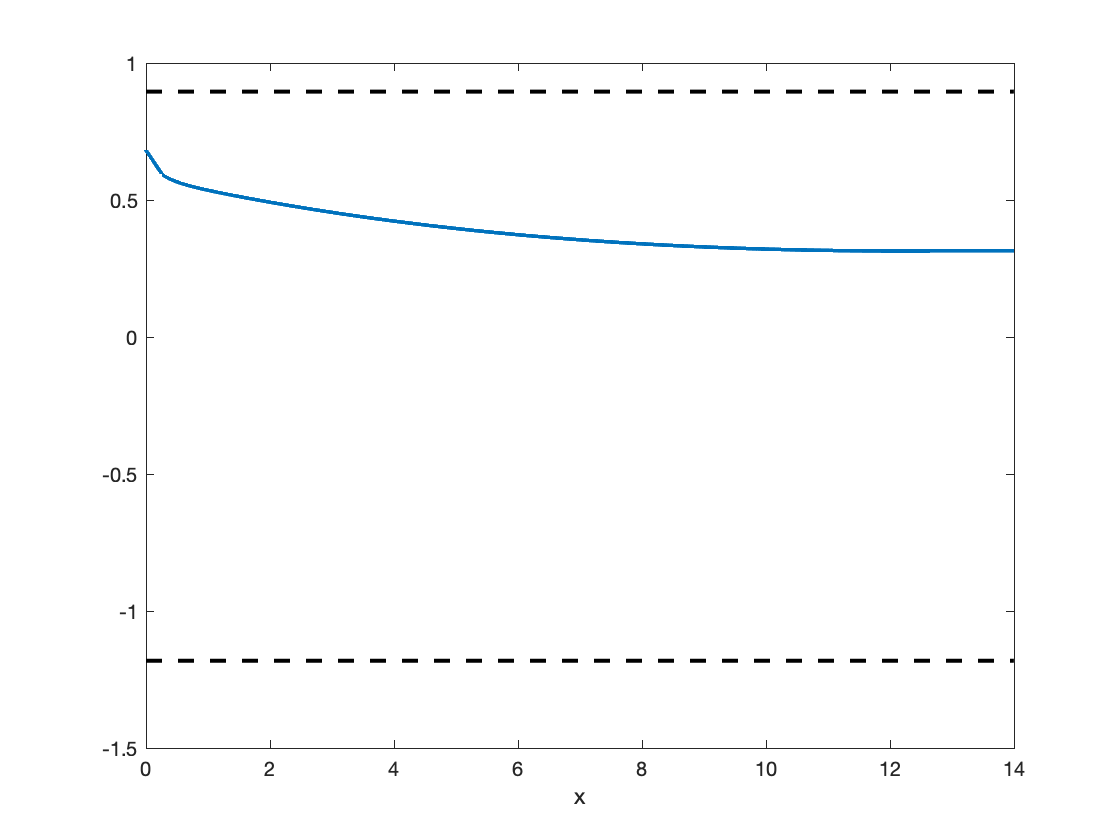}

\caption*{$n=9$}

  \includegraphics[width=.45\textwidth]{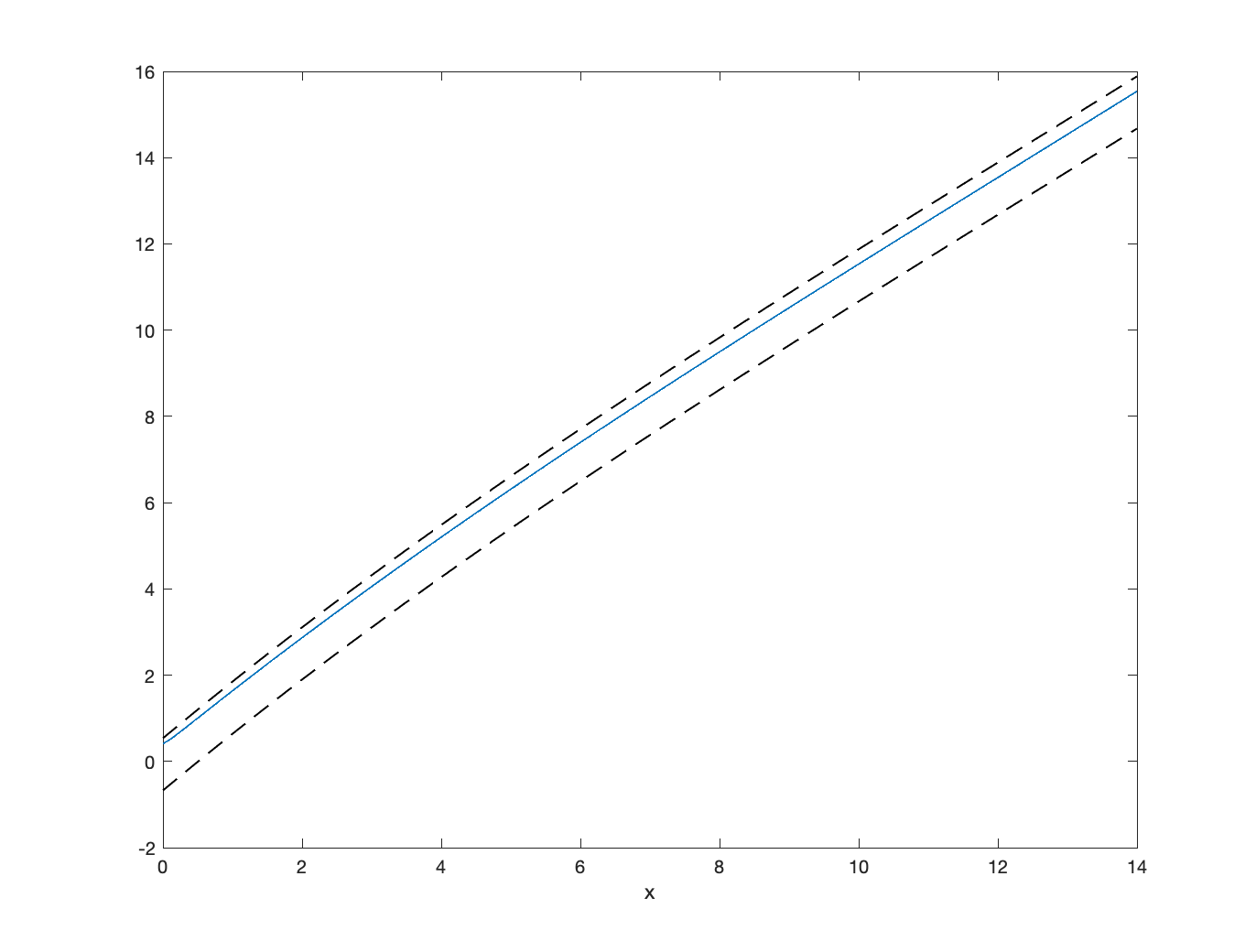}
  \hspace{1cm}
  \includegraphics[width=.45\textwidth]{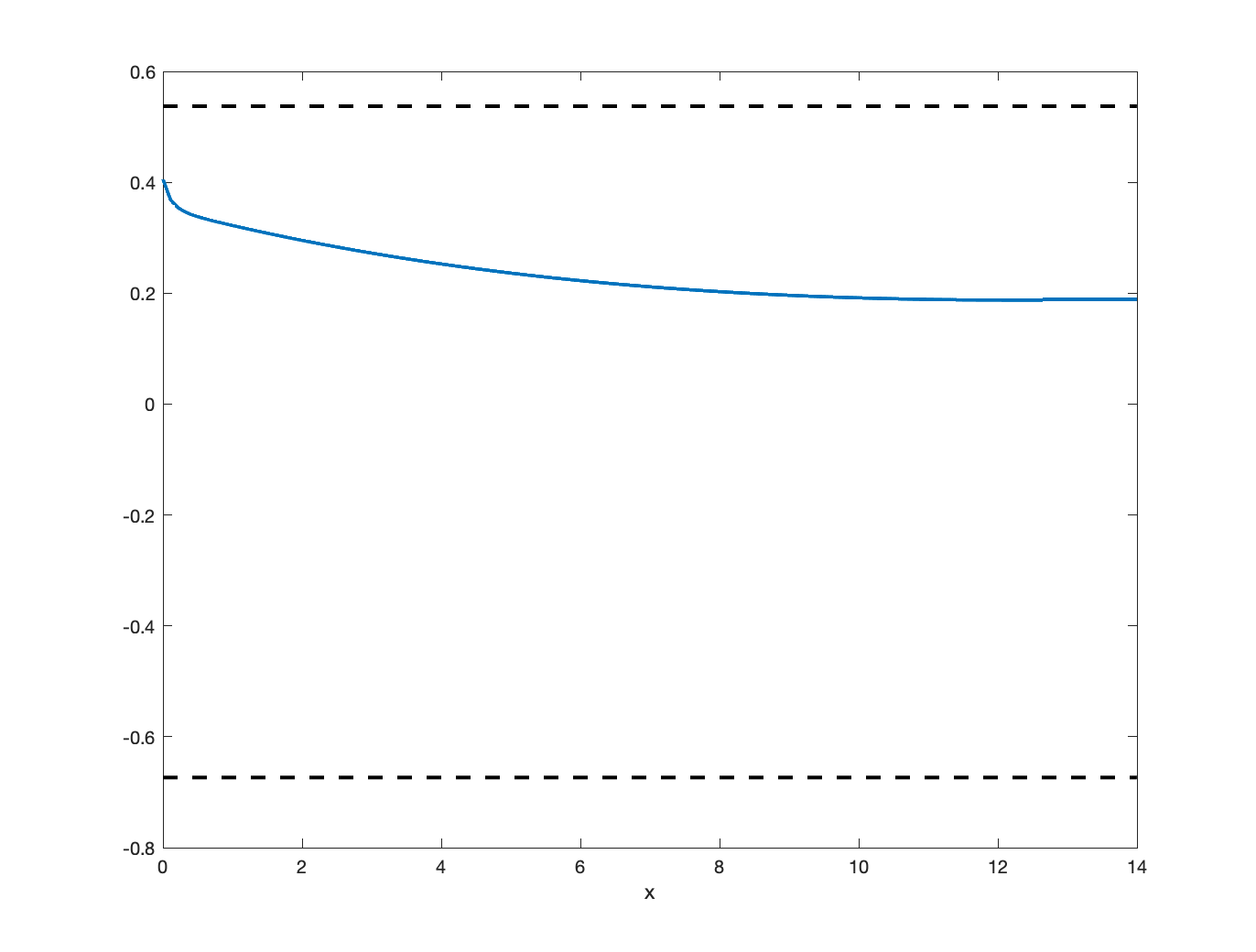}

\caption*{$n=25$}

  \caption{For Example \ref{ex:gamma} the three graphs on the left-hand side are of $V_n$ (solid line) and $V_D+\frac{p}{\sqrtn}$ and $V_D-\frac{q}{\sqrtn}$ (dashed). On the right-hand side the function $V_n-V_D$ is in solid and the dashed boundaries are $\frac{p}{\sqrtn}$ and $-\frac{q}{\sqrtn}$. }\label{fig1}
\end{figure}

\end{document}